\documentclass[12pt]{article}
\usepackage{amsmath, amsfonts, amsthm, amssymb}
\usepackage{centernot}
\usepackage{fullpage}

\newtheorem*{corollary*}{Corollary}
\newtheorem*{definition*}{Definition}
\newtheorem{lemma}{Lemma}
\newtheorem{remark}{Remark}

\newcommand{\QED}{\qed}

\newenvironment{litemize}[1]{%
\begin{list}{}{%
\settowidth{\labelwidth}{#1}%
\setlength{\leftmargin}{\labelwidth}%
\addtolength{\leftmargin}{\labelsep}}}%
{\end{list}}

\newcommand{\lb}{\linebreak}
\newcommand{\cof}{\mathrm{cof\, }}
\newcommand{\RR}{\mathbb{R}}
\newcommand{\eps}{\varepsilon}
\newcommand{\Chi}{\mathcal{X}}
\newcommand{\cD}{\mathcal{D}}
\newcommand{\pd}[2]{\frac{\partial #1}{\partial #2}}
\newcommand{\BBR}[1]{\left( #1 \right)}
\newcommand{\BBS}[1]{\left[ #1 \right]}
\newcommand{\BBA}[1]{\left | #1 \right |}
\newcommand{\mdd}[1]{M^{#1\times #1}}
\newcommand{\mddplus}[1]{M^{#1\times #1}_+}
\newcommand{\PHD}{\frac{\partial{\Phi^A}}{\partial{F_{i\alpha}}}}

\title{{\normalsize \bf CONVERGENCE OF VARIATIONAL APPROXIMATION SCHEMES FOR ELASTODYNAMICS WITH POLYCONVEX ENERGY}\\
{\footnotesize (in: {\it Journal for Analysis and its Applications (ZAA)}, Vol. 33, Issue 1, pp.43-64, 2014)}
}

\author{Alexey Miroshnikov \thanks{ Department of Mathematics, University of Maryland, College Park, USA (amiroshn@gmail.com).} \;   and \, Athanasios E. Tzavaras
\thanks{Department of Applied Mathematics,
University of Crete, Heraklion, Greece; Institute for Applied and Computational Mathematics, FORTH, Heraklion, Greece (tzavaras@tem.uoc.gr).}
}

\date{}

\begin{document}

\maketitle

\abstract{We consider a variational scheme developed by S.~Demoulini, D.~M.~A. Stuart and A.~E. Tzavaras [Arch.\;Rat.\;Mech.\;Anal.\;157\;(2001)]
that approximates the equations of three dimensional
elastodynamics with polyconvex stored energy.
 We establish the convergence of the
time-continuous interpolates constructed in the scheme to a
solution of polyconvex elastodynamics before shock formation. The
proof is based on a relative entropy estimation for the
time-discrete approximants in an environment of $L^p$-theory
bounds, and provides an error estimate for the approximation
before the formation of shocks.}\\

\noindent {\bf keywords}: nonlinear elasticity, polyconvexity, variational approximation scheme. \\
%\end{center}

%\begin{left}
\noindent AMS Subject Classification: 35L70 74B20 74H20

\section{Introduction}

The equations of nonlinear elasticity are the system %%% here alteration:
\begin{equation*}
\label{ELASTINTRO1} y_{tt}=\mathrm{div} \,
\frac{\partial{W}}{\partial{F}}(\nabla{y})
\end{equation*}
where $y:\Omega\times\RR^{+}\to\RR^3$ stands for { the motion}, and we have employed the
constitutive theory of hyperelasticity, i.e.~the Piola-Kirchhoff stress tensor $S$ is
expressed as the gradient, $S(F)=\frac{\partial{W}}{\partial{F}}(F)$, of a stored energy
function $W(F)$. The equations \eqref{ELASTINTRO1} are often recast as a system of
conservation laws,
\begin{equation}\label{ELASTINTRO2}
\begin{aligned}
\partial_{t}v_i &= \partial_{\alpha}  \frac{\partial{W}}{\partial{F_{i\alpha}}}(F)\\
\partial_{t}F_{i\alpha}&=\partial_{\alpha}v_i,
\end{aligned}
\end{equation}
for the velocity %%% here alteration:
$v=\partial_t y$ and the deformation gradient
$F=\nabla y$. The differential constraints %%% here alteration:
\begin{equation*}%\label{GRADINVOLINTRO}
\partial_{\beta} F_{i\alpha} - \partial_{\alpha} F_{i\beta} = 0
\end{equation*}
are propagated from  the kinematic equation
\eqref{ELASTINTRO2}$_2$ and are an involution, \cite{Dafermos86}.

The requirement of frame indifference imposes that
$W(F):M_{+}^{3\times3}\to[0,\infty)$ be invariant under rotations.
This renders the assumption of convexity of $W$ too restrictive
\cite{Tr}, and convexity has been replaced by various weaker
conditions familiar from the theory of elastostatics, see
\cite{Ball77, BCO,Ball02} for a recent survey. A
commonly employed assumption is that of polyconvexity, postulating
that
\begin{equation*}
W(F)=G\circ\Phi(F)
\end{equation*}
where $\Phi(F):=(F,\cof{F},\det{F})$ is the vector of
null-Lagrangians and $G=G(F,Z,w)=G(\Xi)$ is a convex function of
$\Xi\in \RR^{19}$; this encompasses certain physically realistic
models \cite[Section 4.9, 4.10]{Ciarlet}. Starting with the work of
Ball~\cite{Ball77}, substantial progress has been achieved for
handling the lack of convexity of~$W$ within the existence theory
of elastostatics.

For the elastodynamics system local existence of classical
solutions has been established in \cite{DafermosHrusa85},
\cite[Theorem 5.4.4]{Dafermos10} for rank-1 convex stored energies,
and in \cite[Theorem 5.5.3]{Dafermos10} for polyconvex stored
entropies. The existence of global weak solutions is an open
problem, except in one-space dimension, see \cite{DiPerna83}.
Construction of entropic measure valued solutions has been
achieved in \cite{DST} using a variational approximation method
associated to a time-discretized scheme. Various uniqueness
results of smooth solutions in the class of entropy weak and even
dissipative measure valued solutions are available for the
elasticity system \cite{Dafermos86,LT,Dafermos10,DST2}.

The objective of the present work is to show that the
approximation scheme of \cite{DST} converges to the classical
solution of the elastodynamics system before the formation of
shocks. To formulate the problem we outline the scheme in
\cite{DST}  and refer to Section 2 for a detailed presentation.
The null-Lagrangians $\Phi^A(F)$, $A=1,\dots,19$ satisfy
\cite{Qin98} the nonlinear transport identities
\begin{equation*}
\partial_{t}  {\Phi^{A}(F)}={\partial_\alpha }{\biggl(\frac{\partial{\Phi^A}}{\partial{F_{i\alpha}}}(F)v_i}\biggr) \, .
\end{equation*}
This allows to view the system  \eqref{ELASTINTRO2} as constrained
evolution of the extended system
\begin{equation}\label{EXTSYSINTRO}
\begin{aligned}
\partial_{t} v_i &=\partial_{\alpha}\Bigl(\pd {G}{\Xi{_A}}(\Xi)
\,\,\pd{\Phi^{A}}{F_{i\alpha}}(F)\Bigr)\\
\partial_{t}\Xi{_{A}}&=\partial_{\alpha}\Bigl(\pd{\Phi^{A}}{F_{i
\alpha}}(F)\,v_{i}\Bigr). \\
\end{aligned}
\end{equation}
The extension \eqref{EXTSYSINTRO} has the properties: if
$F(\cdot,0)$ is a gradient and $\Xi(\cdot,0)=\Phi(F(\cdot,0))$,
then $F(\cdot,t)$ remains a gradient and
$\Xi(\cdot,t)=\Phi(F(\cdot,t))$ %%% here alteration:
for all $t$. The extended system is
endowed with the entropy identity %%% here alteration:
\begin{equation*}%\label{EXTSYSENTINTRO}
\partial_t\biggl(\frac{|v|^2}{2}+G(\Xi)\biggr)-\partial_{\alpha}\biggl(v_i\,\pd{G}{\Xi_A}(\Xi)\,\pd{\Phi^A}{F_{i\alpha}}(F)\biggr)=0
\end{equation*}
the entropy is convex and the system \eqref{EXTSYSINTRO} is thus
symmetrizable.

{ For periodic solutions $v,\Xi$} (on the torus { $\mathbb{T}^3$}) a variational
approximation method based on the time-discretization of \eqref{EXTSYSINTRO} is proposed
in \cite{DST}: Given a time-step $h>0$ and initial data $(v^0,\Xi^0)$ the scheme provides
the sequence of iterates $(v^j,\Xi^j)$, $j\geqslant 1$, by solving
\begin{equation}\label{DISCEXTSYSINTRO}
\begin{aligned}
\frac{v^j_i-v^{j-1}_i}{h}&=\partial_{\alpha}\Bigl(\pd{G}{\Xi{_{A}}}(\Xi^j)\,\,\pd{\Phi^{A}}{F_{i\alpha}}\,\BBR{F^{j-1}}\Bigr)\\
\frac{\BBR{\Xi^j-\Xi^{j-1}}_{A}}{h}
&=\partial_{\alpha}\Bigl(\pd{\Phi^{A}}{F_{i\alpha}}\,\BBR{F^{j-1}}\,v^j_i\Bigr).
\end{aligned} \quad \text{in } \cD'({\mathbb{T}^3})
\end{equation}
This problem is solvable using variational methods and the
iterates $(v^j,\Xi^j)$ give rise to a  time-continuous approximate
solution $\Theta^{(h)}=(V^{(h)},\Xi^{(h)})$. It is proved in
\cite{DST} that the approximate solution generates a
measure-valued solution of the equations of polyconvex
elastodynamics.
%%\par\smallskip

In this work we consider a smooth solution of the elasticity system \lb
$\bar{\Theta} \!= \!(\bar{V},\bar{\Xi})$ defined on $[0,T] \!\times \!{\mathbb{T}^3}$ and show that
the approximate solution $\Theta^{(h)}$ constructed via the iterates $(v^j,\Xi^j)$ of
\eqref{DISCEXTSYSINTRO}  converges to $\bar{\Theta}=(\bar{V},\bar{\Xi})$ at a convergence
rate $O(h)$. The method of proof is based on the relative entropy method developed for
convex entropies in \cite{Dafermos79,DiPerna79} and adapted for the system of polyconvex
elasticity in \cite{LT} using the embedding to the system \eqref{EXTSYSINTRO}. The
difference between $\Theta^{(h)}$ and $\bar{\Theta}$ is controlled by monitoring the
evolution of the relative entropy
\begin{equation*}
\eta^r=\frac{1}{2}|V^{(h)}-\bar{V}|^2+ G(\Xi^{(h)})-G(\bar{\Xi})
-\nabla{G(\bar{\Xi})}(\Xi^{(h)}-\bar{\Xi}) \,.
\end{equation*}
We establish control of the function
\begin{equation*}
{\mathcal{E}(t)} := \int_{{\mathbb{T}^3}} \,\Bigl( (1+|F^{(h)}|^{p-2}+|\bar{F}|^{p-2})
|F^{(h)}-{\bar{F}}|^2+|\Theta^{(h)}-\bar{\Theta}|^2 \Bigr) \, dx
\end{equation*}
and prove the estimation
\begin{equation*}
\begin{aligned}
{\mathcal{E}(t)} \leqslant C \Bigl( \mathcal{E}(0) +h \Bigr), \quad t\in [0,T]
\end{aligned}
\end{equation*}
which provides the result. There are two novelties in the present
work: (a) In adapting the relative entropy method to the subject
of time-discretized approximations. (b) In employing the method in
an environment where $L^p$-theory needs to be used for estimating
the relative entropy.

{ This work  is a first step towards implementing a finite element method based on the
variational approximation. To do that, one has to devise appropriate finite element
spaces that  preserve the involution structure. This is the subject of a future work.}

The paper is organized as follows. In Section 2 we present the
variational approximation scheme and state the Main Theorem. In
Section 3 we derive the relative entropy identity \eqref{RENTID}
and, finally, in Section 4 we carry out the cumbersome estimations
for the terms in the relative entropy identity and conclude the
proof of Main Theorem via Gronwall's inequality.

\section{The variational approximation scheme and statement of the Main Theorem}

We assume that the stored energy $W:\mddplus{3} \to \RR$ is
\emph{polyconvex}:
\begin{equation}\label{POLYCONVEXITY}
W(F)=G\circ\Phi(F)
\end{equation}
with
\begin{equation*}
G=G(\Xi)=G(F,Z,w):\mdd{3}\times\mdd{3}\times\RR \,\cong\, \RR^{19}
\to \RR
\end{equation*}
uniformly convex and
\begin{equation}\label{PHIDEF}
\Phi(F) = (F, \cof{F}, \det{F}).
\end{equation}
%%\par\medskip

%%\noindent {\bf

\subsection*{Assumptions} We work with periodic boundary
conditions,  i.e.~the spatial domain $\Omega$ is taken to be
{the} three dimensional torus $\mathbb{T}^3$. The indices
$i,\alpha, \dots$ generally run over $1,\dots,3$ while $A,B,\dots$
run over $1,\dots,19$. We use the notation $L^p=L^p(\mathbb{T}^3)$
and $W^{1,p}=W^{1,p}(\mathbb{T}^3)$. Finally, we impose the
following convexity and growth assumptions on $G$:
\begin{litemize}{(H4)}

\item[(H1)] $G\in C^3(\mdd{3}\times \mdd{3} \times \RR;
[0,\infty))$ is of the form
\begin{equation}\label{GDECOMP}
G(\Xi)=H(F) + R(\Xi)
\end{equation}
with $H\in C^3(\mdd{3}; [0,\infty))$ and $R\in C^3(\mdd{3}\times
\mdd{3} \times \RR; [0,\infty))$ strictly convex satisfying
\begin{equation*}
\kappa |F|^{p-2}|z|^{2} \, \leqslant \, z^{T}\nabla^2H(F)z \,
\leqslant \, \kappa' |F|^{p-2}|z|^{2}, \quad \forall z \in \RR^9
\end{equation*}
and $\gamma I  \leqslant  \nabla^2R \leqslant \gamma' I\,$ for some fixed
$\gamma,\gamma',\kappa,\kappa'>0$ and $p\in {[6,\infty)}$.

\item[(H2)] $G(\Xi)\,\geqslant\,c_1|F|^p+c_2|Z|^2+c_3|w|^2-c_4$.

\item[(H3)] $G(\Xi)\,\leqslant\,c_5(|F|^p+|Z|^2+|w|^2+1)$.

\item[(H4)] $|G_{F}|^{\frac{p}{p-1}}+
|G_Z|^{\frac{p}{p-2}}+|G_w|^{\frac{p}{p-3}}
\,\leqslant\,c_6\BBR{|F|^p+|Z|^2+|w|^2+1}.$

\item[(H5)]
$\BBA{\frac{\partial^3H}{\partial{F_{i\alpha}}\partial{F_{ml}}\partial{F_{rs}}}}
\leqslant c_7 |F|^{p-3}$ \quad and \quad
$\BBA{\frac{\partial^3R}{\partial{\Xi_A}
\partial{\Xi_B} \partial{\Xi_D}}} \leqslant c_8$.

\end{litemize}

%%\par\medskip

%%\noindent {\bf

\subsection*{Notations} To simplify notation we write
\begin{equation*}
\begin{aligned}
G_{,A}\,(\Xi)&=\pd{G}{\Xi_{A}}(\Xi),& \quad
R_{,A}\,(\Xi)&=\pd{R}{\Xi_{A}}(\Xi),&\\
H_{,i\alpha}\,(F)&=\pd{H}{F_{i\alpha}}(F),& \quad
\Phi^{A}_{,i\alpha}\,(F)&=\pd{\Phi^{A}}{F_{i\alpha}}(F).&
\end{aligned}
\end{equation*}
In addition, for each $i,\alpha =1,2,3$ we set
\begin{equation}\label{GCOMPFLD}
\begin{aligned}
g_{i\alpha}(\Xi,F^*)=\pd{G}{\Xi_{A}}\,(\Xi)
\,\pd{\Phi^{A}}{F_{i\alpha}}\,(F^*), \quad F^* \in \RR^9, \,\Xi
\in \RR^{19}
\end{aligned}
\end{equation}
(where we use {the} summation convention over repeated indices)
and denote the corresponding fields $g_i:
\RR^{19}\times\RR^9\to\RR^3$ by %%% here alteration:
\begin{equation*}%\label{GFLD}
g_i(\Xi,F^*):=(g_{i1},g_{i2},g_{i3})(\Xi,F^*).
\end{equation*}

\subsection{Time-discrete variational scheme}

The equations of elastodynamics \eqref{ELASTINTRO1} for polyconvex
stored-energy \eqref{POLYCONVEXITY} can be expressed as a system
of conservation laws,
\begin{equation}\label{PXELASTSYS}
\begin{aligned}
\partial_{t} v_i&=\partial_{\alpha}\biggl(\pd {G}{\Xi_{A}}(\Phi(F))
\pd{\Phi^{A}}{F_{i\alpha}}(F)\biggr)\\
\partial_{t} F_{i\alpha}&=\partial_{\alpha} v_i
\end{aligned}
\end{equation}
which is equivalent to \eqref{ELASTINTRO1} subject to differential
constrains
\begin{equation}\label{GRADCONSTR}
\partial_{\beta} F_{i\alpha} - \partial_{\alpha} F_{i\beta} = 0
\end{equation}
that are an involution \cite{Dafermos86}: if they are satisfied
for $t=0$ then \eqref{PXELASTSYS} propagates \eqref{GRADCONSTR} to
satisfy for all times. Thus the system \eqref{PXELASTSYS} is
equivalent to systems \eqref{ELASTINTRO1} whenever $F(\cdot,0)$ is
a gradient.

%%\par\smallskip

The components of $\Phi(F)$ defined by \eqref{PHIDEF} are
null-Lagrangians and satisfy %%% here alteration:
\begin{equation*}%e\label{SNLP}
\partial_{\alpha}\biggl(\pd{\Phi^A}{F_{i\alpha}}(\nabla{u})\biggr) =
0, \quad A=1,\dots,19
\end{equation*}
for any smooth $u(x):\RR^3 \to\RR^3$. Therefore, if $(v,F)$ are
smooth solutions of~\eqref{PXELASTSYS}, the null-Lagrangians
$\Phi^A(F)$ satisfy the transport identities \cite{DST}
\begin{equation}\label{NLTID}
\partial_t \Phi^A(F) = \partial_{\alpha} \biggl(\pd{\Phi^A}{F_{i\alpha}}(F)v_i
\biggr), \quad \forall F \  \mbox{with} \
\partial_{\beta}F_{i\alpha}=\partial_{\alpha}F_{i\beta}.
\end{equation}
Due to the identities \eqref{NLTID} the system of polyconvex
elastodynamics \eqref{PXELASTSYS} can be embedded into the
enlarged system \cite{DST}
\begin{equation}\label{EXTSYS}
\begin{aligned}
\partial_{t}v_i&=\partial_{\alpha}\biggl(\pd {G}{\Xi_A}(\Xi)
\,\,\pd{\Phi^{A}}{F_{i\alpha}}(F)\biggr)\\
\partial_{t}\Xi_{A}&=\partial_{\alpha}\biggl(\pd{\Phi^{A}}{F_{i
\alpha}}(F)\,v_{i}\biggr).
\end{aligned}
\end{equation}
The extension has the following properties:
\begin{litemize}{(E\,3)}

\item[(E\,1)] If $F(\cdot,0)$ is a gradient then $F(\cdot,t)$
remains a gradient %%% here alteration:
for all $t$.

\item[(E\,2)] If $F(\cdot,0)$ is a gradient and
$\Xi(\cdot,0)=\Phi(F(\cdot,0))$, then $F(\cdot,t)$ remains a
gradient and $\Xi(\cdot,t)=\Phi(F(\cdot,t))$ %%% here alteration:
for all $t$. In
other words, the system of polyconvex elastodynamics can be viewed
as a constrained evolution of \eqref{EXTSYS}.

\item[(E\,3)] The enlarged system admits a convex entropy
\begin{equation}\label{ENTDEF}
\eta(v,\Xi) = \tfrac{1}{2} |v|^2 +G(\Xi), \quad (v,\Xi) \in
\RR^{22}
\end{equation}
and thus is symmetrizable (along the solutions that are
gradients).
\end{litemize}

%%\par\smallskip

Based on the time-discretization of the enlarged system
\eqref{EXTSYS} S.~Demoulini, D.~M.~A.~Stuart and A.~E.~Tzavaras \cite{DST}
developed a variational approximation scheme which, for the given
initial data %%% here alteration:
\begin{equation*}%\label{INITITERCLASS}
\Theta^0:=(v^0, \Xi^0)=(v^0,F^0,Z^0,w^0) \in L^2 \times L^p \times
L^2 \times L^2
\end{equation*}
and fixed $h>0$, constructs the sequence of successive iterates %%% here alteration
\begin{equation*}%\label{GENITERCLASS}
\Theta^j:=(v^j, \Xi^j)=(v^j,F^j,Z^j,w^j)  \in  L^2 \times L^p
\times L^2 \times L^2, \quad j \geqslant 1
\end{equation*}
with the following properties (see \cite[Lemma 1, Corollary
2]{DST}):

\begin{litemize}{(P\,5)}

\item[(P\,1)] The iterate $(v^j,\Xi^j)$ is the unique minimizer of
the functional
\begin{equation*}
\mathcal{J}(v,\Xi) = \int_{\mathbb{T}^3}
\biggl(\tfrac{1}{2}|v-v^{j-1}|^2 + G(\Xi)\biggr) dx
\end{equation*}
over the weakly closed affine subspace
\begin{equation*}
\begin{aligned}
\mathcal{C} = &\biggl\{ (v,\Xi)\in L^2 \times L^p \times L^2
\times
L^2: \ \mbox{such that} \ \forall \varphi \in C^{\infty}(\mathbb{T}^3)\\
& \, \int_{\mathbb{T}^3} \biggl(\frac{\Xi_{A} -
\Xi^{j-1}_{A}}{h}\biggr) \varphi \,dx = - \int_{\mathbb{T}^3}
\biggl(\pd{\Phi^{A}}{F_{i\alpha}}(F^{j-1}) v_i\biggr) \partial_{\alpha} \varphi \, dx \biggr\}.
\end{aligned}
\end{equation*}

\item[(P\,2)] For each $j\geqslant 1$ the iterates satisfy
\begin{equation}\label{DISCEXTSYS}
\begin{aligned}
\frac{v^j_i-v^{j-1}_i}{h}&=\partial_{\alpha}\biggl(\pd{G}{\Xi_{A}}(\Xi^j)\,\,\pd{\Phi^{A}}{F_{i\alpha}}(F^{j-1})\biggr)\\
\frac{\Xi^j_A-\Xi^{j-1}_A}{h}&=\partial_{\alpha}\biggl(\pd{\Phi^{A}}{F_{i\alpha}}(F^{j-1})\,v^j_i\biggr)
\end{aligned} \quad \mbox{in} \ \cD'(\mathbb{T}^3).
\end{equation}

\item[(P\,3)] If $F^0$ is a gradient, then so is $F^j\,$ %%% here alteration:
 for all  $j \geqslant 1$.

\item[(P\,4)] Iterates $v^j$, $j\geqslant 1$ have higher
regularity: $v^j \in W^{1,p}(\mathbb{T}^3)$ %%% here alteration:
for all $j\geqslant 1$.

\item[(P\,5)]There exists $E_0>0$ determined by the initial data
such that
\begin{equation}\label{ITERBOUND}
\begin{aligned}
\sup_{j \geqslant \, 0} \Bigl( \| v^j\|^2_{L^2_{dx}}
+\int_{\mathbb{T}^3} G(\Xi^j) \,dx \Bigr) + \sum_{j=1}^{\infty}
\|\Theta^j - \Theta^{j-1}\|_{L^2_{dx}}^2  \leqslant E_0.
\end{aligned}
\end{equation}

\end{litemize}

%%\par\smallskip

Given the sequence of spatial iterates $(v^j,\Xi^j)$, $j\geqslant
1$ we define  (following~\cite{DST}) the time-continuous,
piecewise linear interpolates $\Theta^{(h)}:=(V^{(h)}, \Xi^{(h)})$
by
\begin{equation}\label{CONTINTP}
\begin{aligned}
V^{(h)}(t)&=\sum^{\infty}_{j=1}\Chi^j(t) \Bigl(v^{j-1}+\frac{t-h(j-1)}{h}(v^j-v^{j-1})\Bigr)\\
\Xi^{(h)}(t)&=\bigl(F^{(h)},Z^{(h)},w^{(h)}\bigr)(t)\\
&=\sum^{\infty}_{j=1}\Chi^j(t)\Bigl(
\Xi^{j-1}+\frac{t-h(j-1)}{h}(\Xi^j - \Xi^{j-1})\Bigr),
\end{aligned}
\end{equation}
and the piecewise constant interpolates $\theta^{(h)}:=(v^{(h)},
\xi^{(h)})$ and $\tilde{f}^{(h)}$ by
\begin{equation}\label{CONSTINTP}
\begin{aligned}
v^{(h)}(t)&=\sum^{\infty}_{j=1}\Chi^j(t)v^j\\
\xi^{(h)}(t)&=(f^{(h)},z^{(h)},\omega^{(h)})(t)=\sum^{\infty}_{j=1}\Chi^j(t)\Xi^j\\
\tilde{f}^{(h)}(t)&=\sum^{\infty}_{j=1}\Chi^j(t)F^{j-1},
\end{aligned}
\end{equation}
where $\Chi^j(t)$ is the characteristic function of the interval
$I_j:=[(j-1)h,jh)$. Notice that $\tilde{f}^{(h)}$ is the
time-shifted version of $f^{(h)}$  and it is used later in
defining a relative entropy flux, as well as the time-continuous
equations \eqref{TCDSYS}.

%%\par\smallskip

Our main objective is to prove convergence of the interpolates
$(V^{(h)},F^{(h)})$ obtained via the variational scheme to the
solution of polyconvex elastodynamics as long as the limit
solution remains smooth. This is achieved by employing the
extended system \eqref{EXTSYS} and proving convergence of the
time-continuous approximates $\Theta^{(h)}=(V^{(h)},\Xi^{(h)})$ to
the solution $\bar{\Theta}=(\bar{V},\bar{\Xi})$ of the extension
\eqref{EXTSYS} as long as $\bar{\Theta}$ remains smooth.

\newtheorem*{MThm}{Main Theorem}
\begin{MThm}
Let $W$ be defined by \eqref{POLYCONVEXITY} with  $G$ satisfying
{\rm (H1)--(H5)}. Let $\Theta^{(h)}=(V^{(h)}, \Xi^{(h)})$,
$\theta^{(h)}=(v^{(h)}, \xi^{(h)})$ and $\tilde{f}^{(h)}$ be the
interpolates defined via \eqref{CONTINTP}, \eqref{CONSTINTP} and
induced by the sequence of spatial iterates %%% here alteration
\begin{equation*}
\Theta^j=(v^j,\Xi^j)=(v^j,F^j,Z^j,w^j) \in L^2 \times L^p \times
L^2 \times L^2, \quad j \geqslant 0
\end{equation*}
which satisfy {\rm (P1)--(P5)}. Let $\bar{\Theta}=(\bar{V}, \bar{\Xi})=(\bar{V},
\bar{F},\bar{Z},\bar{w})$ be { the} smooth solution of \eqref{EXTSYS} defined on
$\mathbb{T}^3 \times [0,T]$ and emanate from the data $\bar{\Theta}^0=(\bar{V}^0,
\bar{F}^0,\bar{Z}^0,\bar{w}^0)$. Assume also that ${F}^0,\bar{F}^0$ are gradients. Then:
\begin{itemize}
\item[\rm (a)] The relative entropy
$\eta^r=\eta^r(\Theta^{(h)},\bar{\Theta})$ defined by
\eqref{RENTDEF} satisfies \eqref{RENTID}. Furthermore, there exist
constants $\mu,\mu'>0$ such that
\begin{equation*}
\mu \, \mathcal{E}(t) \leqslant \int_{\mathbb{T}^3} \,\eta^r(x,t)\, dx \leqslant \mu'
\mathcal{E}(t), \quad t\in [0,T]
\end{equation*}
 where
\begin{equation*}
\mathcal{E}(t) := \int_{\mathbb{T}^3} \,\Bigl(
(1+|F^{(h)}|^{p-2}+|\bar{F}|^{p-2})
|F^{(h)}-{\bar{F}}|^2+|\Theta^{(h)}-\bar{\Theta}|^2 \Bigr) \, dx.
\end{equation*}
\item[\rm (b)] There exists $\varepsilon>0$ and
$C=C(T,\bar{\Theta},E_0,\mu,\mu',\varepsilon)>0$ such that %%% here alteration:
for all $h\in(0,\varepsilon)$
\begin{equation*}
\begin{aligned}
\mathcal{E}(\tau) \leqslant C \,\bigl(\mathcal{E}(0) +h \bigr),
\quad \tau\in [0,T]  .
\end{aligned}
\end{equation*}
Moreover, if the data satisfy $\,\mathcal{E}^{(h)}(0) \to 0$ as
$h\downarrow 0$, then
\begin{equation*}
\sup_{t\in[0,T]}\int_{\mathbb{T}^3} \BBR{|\Theta^{(h)} -
\bar{\Theta}|^2 +
|F^{(h)}-\bar{F}|^2(1+|F^{(h)}|^{p-2}+|\bar{F}|^{p-2})}dx
\rightarrow 0
\end{equation*}
as $h\downarrow 0$.
\end{itemize}
\end{MThm}

\begin{corollary*}
Let $\Theta^{(h)}=(V^{(h)},\Xi^{(h)})$ be as in the Main Theorem.
Let $(\bar{V},\bar{F})$ be a smooth solution of \eqref{PXELASTSYS}
with $\bar{F}(\cdot,0)$ a gradient and
$\bar{\Theta}=(\bar{V},\Phi(\bar{F}))$. Assume that initial data
satisfy $\Theta^{(h)}(\cdot,0)=\bar{\Theta}(\cdot,0)$. Then
\begin{equation*}
 \sup_{t\in[0,T]}  \BBR{ \|V - \bar{V} \|^2_{L^2(\mathbb{T}^3)} + \|\Xi^{(h)} - \Phi(\bar{F}) \|^2_{L^2(\mathbb{T}^3)}  +
\|F^{(h)}-\bar{F}\|^p_{L^p(\mathbb{T}^3)} } = O(h).
\end{equation*}
\end{corollary*}

\begin{remark}
The smooth solution $\bar{\Theta}=(\bar{V},\bar{\Xi})$ to the extended system
\eqref{EXTSYSINTRO} is provided beforehand. A natural question arises whether such a
solution exists. We briefly discuss the existence theory for  \eqref{ELASTINTRO2} on the
torus $\mathbb{T}^3$. In \cite{DafermosHrusa85} energy methods are used to establish
local (in time) existence of smooth solutions to certain initial-boundary value problem
that apply to the system of nonlinear elastodynamics \eqref{ELASTINTRO1} with rank-1
convex stored energy. More precisely, for a bounded domain $\Omega\subset \RR^n$ with the
smooth boundary $\partial \Omega$ the authors establish (\cite[Theorem
5.2]{DafermosHrusa85}) the existence of the unique {motion} $y( \cdot,t)$ satisfying
\eqref{ELASTINTRO1} in $\Omega \times [0,T]$ together with boundary conditions $y(x,t)=0$
on $\partial\Omega \times [0,T]$ and initial conditions $y(\cdot,0)=y_0$ and
$y_t(\cdot,0)=y_1$ whenever $T>0$ is small enough and the initial data lie in a compact
set. One may get a counterpart of this result for solutions on $\mathbb{T}^3$ since the
methods in \cite{DafermosHrusa85} are developed in the abstract framework: a quasi-linear
partial differential equation is viewed as an abstract differential equation with initial
value problem set on an interpolated scale of separable Hilbert spaces
${\{H_{\gamma}\}}_{\gamma\in[0,m]}$ with $m \geqslant 2$. To be precise, the spaces
satisfy $H_{\gamma}=[H_0, H_m]_{\gamma/m}$ and the desired solution $u(t)$ of an abstract
differential equation is assumed to be taking values in $H_m \bigcap V$, where $V$, a
closed subspace of $H_1$, is designated to accommodate the boundary conditions
(cf.~\cite[Section~2]{DafermosHrusa85}). By choosing appropriate spaces, namely
\begin{equation*}
    H_{\gamma} = \left[L^2(\mathbb{T}^3), W^{m,2}(\mathbb{T}^3)
    \right]_{\gamma/m}  \quad \mbox{and} \quad
    V=H_1=W^{1,2}(\mathbb{T}^3),
\end{equation*}
and requiring strong ellipticity (cf.~\cite[Section 5]{DafermosHrusa85}) for the stored energy one may apply
\cite[Theorem 4.1]{DafermosHrusa85} to conclude the local existence of
smooth solutions on the torus $\mathbb{T}^3$ to the system of
elastodynamics \eqref{ELASTINTRO1} and hence to
\eqref{ELASTINTRO2}. Since strong polyconvexity implies strong
ellipticity \cite{Ball77}, the same conclusion holds for the case
of polyconvex energy which is used here.
\end{remark}

\begin{remark}
The framework for existence of measure-valued solutions for the polyconvex elasticity system
(see (H1)--(H4) of \cite{DST}) and that of uniqueness of classical within the class of measure-valued
solutions (see \cite{DST2}) is more general than the framework used in the Main Theorem.
This discrepancy is due to the relative entropy being  best adapted to an $L^2$ setting
and technical difficulties connected to the estimations of the time-step approximants of \eqref{DISCEXTSYS}.
Our approach,  based on using the "distance" function in \eqref{DISTDEF} as a substitute for
the relative entropy, simplifies the estimations but  limits applicability to stored energies
\eqref{POLYCONVEXITY}, \eqref{GDECOMP} with
 $L^p$-growth for $F$ but only $L^2$-growth in $\cof F$ and  $\det F$.
% In return, an error estimate of $O(h)$ is  provided in this context.
\end{remark}

\section{Relative entropy identity}

For the rest of the sequel, we suppress the dependence on $h$ to
simplify notations and, {\it cf.} Main Theorem, assume:
\begin{itemize}
\item[{(1)}] $\Theta=(V,\Xi)$, $\theta=(v,\xi)$, $\tilde{f}$ are the approximates defined
by \eqref{CONTINTP} and \eqref{CONSTINTP}.

\item[{(2)}]
$\bar{\Theta}=(\bar{V},\bar{\Xi})=(\bar{V},\bar{F},\bar{Z},\bar{w})$
is a smooth solution of \eqref{EXTSYS} defined on \lb
$\mathbb{T}^3 \times [0,T]$ where $T>0$ is finite.
\end{itemize}

%%\par\smallskip

The goal of this section is to derive an identity for a relative
energy among the two solutions. To this end, we define the
relative entropy
\begin{equation}\label{RENTDEF}
\eta^r(\Theta,\Bar{\Theta}):=\eta(\Theta)-\eta(\bar{\Theta})-\nabla\eta(\bar{\Theta})(\Theta-\bar{\Theta})
\end{equation}
and the associated relative flux which will turn out to be
\begin{equation} \label{RENTFLUX}
\begin{aligned}
q^r_{\alpha}(\theta,\Bar{\Theta},\tilde{f}):=(v_i-\bar{V}_i)\bigl(G_{,A}(\xi)-G_{,A}(\bar{\Xi})\bigr)
\, \Phi_{,i\alpha}^A(\tilde{f}), \quad \alpha = 1,2,3.
\end{aligned}
\end{equation}

%%\par\smallskip

We now state two elementary lemmas used in our further
computations. The first one extends the null-Lagrangian properties
while the second one provides the rule for the divergence of the
product in the non-smooth case.

\begin{lemma}[null-Lagrangian properties]\label{DIVPHIZLMM}
Assume $q>2$ and $r \geqslant \tfrac{q}{q-2}$. Then, if $u \in
W^{1,q}(\mathbb{T}^3;\RR^3)$, $z\in W^{1,r}(\mathbb{T}^3)$, we
have
\begin{equation*}
\begin{aligned}
\partial_{\alpha}\biggl(\PHD\BBR{\nabla{u}}\biggr)&=0\\
\partial_{\alpha}\biggl(\PHD(\nabla{u})z\biggr) &= \PHD(\nabla{u})\,\partial_{\alpha}z
\end{aligned} \quad \mbox{in} \ \cD'(\mathbb{T}^3)
\end{equation*}
for each $i=1,\dots,3$ and $A=1,\dots,19$.
\end{lemma}

%%\par\smallskip

\begin{lemma}[product rule]\label{DIVPDLMM}
Let $q \in (1,\infty)$ and $q'=\frac{q}{q-1}$. Assume
\begin{equation*}
f\in W^{1,q}(\mathbb{T}^3),\  h\in L^{q'}(\mathbb{T}^3;\RR^3)
\quad \mbox{and} \quad \mathrm{div} \,h \in L^{q'}(\mathbb{T}^3).
\end{equation*}
Then $fh\in L^1(\mathbb{T}^3;\RR^3)$, $\mathrm{div}\,(fh)\in
L^1(\mathbb{T}^3)$ and
\begin{equation*}
\mathrm{div}\,(fh)=f\mathrm{div}\,h+\nabla{f}h \quad \mbox{in}
\ \cD'(\mathbb{T}^3).
\end{equation*}
\end{lemma}

%%\par\smallskip

\begin{lemma}[relative entropy identity]\label{RENTIDLMM}
For almost all $t\in [0,T]$
\begin{equation}\label{RENTID}
\partial_t \eta^r-\mathrm{div}\,q^r
=Q-\frac{1}{h} \sum_{j=1}^{\infty} \Chi^j(t) D^j+S \quad \mbox{in}
\ \cD'(\mathbb{T}^3)
\end{equation}
where
\begin{equation}\label{QTERM}
\begin{aligned}
Q &:= \partial_{\alpha}(G_{,A}(\bar{\Xi}))
\bigl(\Phi_{,i\alpha}^A(F)-\Phi_{,i\alpha}^A(\bar{F})\bigr)
\bigl(V_i-\bar{V}_i\bigr)\\[1pt]
&\phantom{:=}\; +\partial_{\alpha}\bar{V}_i
\bigl(G_{,A}(\Xi)-G_{,A}(\bar{\Xi})\bigr)
\bigl(\Phi_{,i\alpha}^A(F)-\Phi_{,i\alpha}^A(\bar{F})\bigr)\\[1pt]
&\phantom{:=}\;+\partial_{\alpha}\bar{V}_i
\bigl(G_{,A}(\Xi)-G_{,A}(\bar{\Xi})-G_{,AB}(\bar{\Xi})(\Xi-\bar{\Xi})_B
\bigr)\Phi_{,i\alpha}^A(\bar{F})
\end{aligned}
\end{equation}
estimates the difference between the two solutions,
\begin{equation}\label{DTERM}
D^j:=\bigl(\nabla{\eta}(\theta)-\nabla{\eta}(\Theta)\bigr)
\delta\Theta^j,
\end{equation}
where $\delta\Theta^j:=\Theta^j - \Theta^{j-1}$, are the
dissipative terms, and
\begin{equation}\label{STERM}
\begin{aligned}
S := & \, \partial_{\alpha}(G_{,A}(\bar{\Xi}))\Bigl[ \,\Phi_{,i\alpha}^A(\bar{F})               \bigl(v_i-V_i\bigr)
%%\\ &\!\quad\qquad\qquad\qquad
+  \bigl(\Phi_{,i\alpha}^A(F)-\Phi_{,i\alpha}^A(\bar{F})\bigr)    \bigl(v_i-V_i\bigr)\\
%%&\!\quad\qquad\qquad\qquad
&+ \bigl(\Phi_{,i\alpha}^A(\tilde{f})-\Phi_{,i\alpha}^A(F)\bigr)   \bigl(v_i-V_i\bigr)
%%\\ &\!\quad\qquad\qquad\qquad
+ \bigl(\Phi_{,i\alpha}^A(\tilde{f})-\Phi_{,i\alpha}^A(F)\bigr)   \bigl(V_i-\bar{V_i}\bigr) \Bigr]\\
&+\partial_{\alpha}\bar{V_i}\Bigl[ \bigl(G_{,A}(\xi)-G_{,A}(\Xi) \bigr)\Phi_{,i\alpha}^A(\bar{F})\\
%%\\ &\!\!\!\qquad\qquad\qquad
&+ \bigl(G_{,A}(\xi)-G_{,A}(\Xi)\bigr)        \bigl(\Phi_{,i\alpha}^A(\tilde{f})-\Phi_{,i\alpha}^A(F)\bigr)\\
%%\\ &\!\!\!\qquad\qquad\qquad
&+ \bigl(G_{,A}(\xi)-G_{,A}(\Xi)\bigr)        \bigl(\Phi_{,i\alpha}^A(F)-\Phi_{,i\alpha}^A(\bar{F})\bigr)  \\
%%&\!\!\!\qquad\qquad\qquad
&+ \bigl(G_{,A}(\Xi)-G_{,A}(\bar{\Xi})\bigr)  \bigl(\Phi_{,i\alpha}^A(\tilde{f})-\Phi_{,i\alpha}^A(F)\bigr)\Bigr]
\end{aligned}
\end{equation}
is the error term.
\end{lemma}

\begin{proof}
Notice that by \eqref{CONTINTP} for almost all $t \geqslant 0$
\begin{equation}\label{TIMEDVVXI}
\begin{aligned}
\partial_{t}V(\cdot,t) &=\sum_{j=1}^{\infty}\Chi^j(t)\frac{\delta v^j}{h}, \quad \delta v^j:={v^j-v^{j-1}}\\
\partial_{t}\Xi(\cdot,t)&=\sum_{j=1}^{\infty}\Chi^j(t)\frac{\delta
\Xi^j}{h},\quad \delta \Xi^j:=\Xi^j - \Xi^{j-1}.
\end{aligned}
\end{equation}
Hence by \eqref{GCOMPFLD}, \eqref{DISCEXTSYS} and
\eqref{TIMEDVVXI} we obtain for almost all $t \geqslant 0$
\begin{equation}\label{TCDSYS}
\begin{aligned}
\partial_{t} V_i(\cdot,t)&=\mathrm{div}
\bigl( g_i(\xi,\tilde{f})\bigr)\\[0pt]
\partial_{t} \Xi_{A}(\cdot,t)&=\partial_{\alpha}\bigl( \hspace{1pt} \Phi^A_{,i\alpha}(\tilde{f})\,v_i\bigr)
\end{aligned}\quad \mbox{in} \ \cD'(\mathbb{T}^3).
\end{equation}
Since $(\bar{V},\bar{\Xi})$ is the smooth solution of
\eqref{EXTSYS}, using \eqref{GCOMPFLD} we also have
\begin{equation}\label{EXTSYS2}
\begin{aligned}
\partial_{t}\bar{V}_i &= \mathrm{div} \bigl( \hspace{1pt} g_i(\bar{\Xi},\bar{F}) \bigr)
\\
\partial_{t}\bar{\Xi}_{A}&=\partial_{\alpha}\bigl(\Phi^A_{,i\alpha}(\bar{F})\,\bar{V}_{i}\bigr)
\end{aligned} \quad \mbox{in} \ \mathbb{T}^3 \times [0,T].
\end{equation}

%%\par\smallskip

Further in the proof we will perform a series of calculations that
hold for smooth functions. A technical difficulty arises, since
the iterates $(v^j,\Xi^j)$, $j\geqslant 1$ satisfying
\eqref{DISCEXTSYS} are, in general, not smooth. To bypass this we
employ Lemmas~\ref{DIVPHIZLMM} and \ref{DIVPDLMM} that provide the
null-Lagrangian property and product rule in the smoothness class
appropriate for the approximates $\Theta \!=\! (V,\Xi)$,
$\theta\!= \!(v,\xi)$, $\tilde{f}$.

%%\par\smallskip

By assumption $F^0$ and $\bar{F}^0$ are gradients. Hence using
(P\,3) we conclude that $F^j$, $j\geqslant 1$ are gradients.
Furthermore, from (E1) it follows that $\bar{F}$ remains a
gradient %%% here alteration:
 for all $t$. Thus, recalling
\eqref{CONTINTP}, \eqref{CONSTINTP}, we have
\begin{equation}\label{FGRADPROP}
\mbox{$F$, $f$, $\tilde{f}$ and $\bar{F}$ are gradients %%% here alteration:
for all $t
\in [0,T]$}.
\end{equation}
We also notice that by \eqref{PHIDEF}, \eqref{GCOMPFLD}, and (H4)
we have for all $F^{*}\in \RR^9$, $\Xi^{\circ}\in \RR^{19}$
    \begin{equation}\label{gEST}
\begin{aligned}
\bigl|g_{i\alpha}\bigl(\Xi^{\circ},F^*\bigr)\bigr|^{p'}
%%\\[1pt]
&\leqslant
C_g\Bigl(\,\Bigl|\pd{G}{F_{i\alpha}}\Bigr|^{\frac{p}{p-1}}+\bigl|F^*\bigr|^{\frac{p}{p-1}}
\Bigl|\pd{G}{Z_{k\gamma}}\Bigr|^{\frac{p}{p-1}}+\bigl|F^*\bigr|^{\frac{2p}{p-1}}
\Bigl|\pd{G}{w}\Bigr|^{\frac{p}{p-1}}\Bigr)\\[1pt]
&\leqslant
C_g'\Bigl(\,|F^*|^p+\Bigl|\pd{G}{F_{i\alpha}}\Bigr|^{\frac{p}{p-1}}+
\Bigl|\pd{G}{Z_{k\gamma}}\Bigr|^{\frac{p}{p-2}}+\Bigl|\pd{G}{w}\Bigr|^{\frac{p}{p-3}}\Bigl)\\[1pt]
&\leqslant
C_g''\Bigl(|F^*|^p+|F^{\circ}|^p+|Z^{\circ}|^2+|w^{\circ}|^2+1\Bigr)
\end{aligned}
\end{equation}
where $p\in{[6,\infty)}$ and $p'=\tfrac{p}{p-1}$. Hence (H2), (P4)--(P5),
\eqref{CONSTINTP}$_1$ and Lemmas~\ref{DIVPHIZLMM}, \ref{DIVPDLMM} along with
\eqref{TCDSYS}$_1$ imply
\begin{equation}\label{DIVPRODUCTS}
\begin{aligned}
\mathrm{div} \bigl(v_ig_i(\xi,\tilde{f})\bigr) &=
{v_i}{\partial_{t} V_i} + \nabla{v_i}
g_i(\xi,\tilde{f})\\
\mathrm{div}\bigl(\bar{V}_i g_i(\xi,\tilde{f})\bigr) &= \bar{V}_i
\partial_t V_i + \nabla{\bar{V}_i} g_i\,(\xi,\tilde{f})\\
\mathrm{div}\bigl(v_ig_i(\bar{\Xi},\tilde{f})\bigr) &= v_i
\Phi_{,i\alpha}^A(\tilde{f})\,\partial_\alpha(G_{,A}(\bar{\Xi}))+\nabla{v_i}
g_i(\bar{\Xi},\tilde{f})\\
\mathrm{div}\bigl(\bar{V}_ig_i(\bar{\Xi},\tilde{f})\bigr) &=
\bar{V}_i
\Phi_{,i\alpha}^A(\tilde{f})\,\partial_\alpha(G_{,A}(\bar{\Xi}))+\nabla{\bar{V}_i}
g_i(\bar{\Xi},\tilde{f}).\\
\end{aligned}
\end{equation}
Similarly, by (P4), Lemma \ref{DIVPHIZLMM}, \eqref{TCDSYS}$_2$ and
\eqref{FGRADPROP} we have the identity
\begin{equation}\label{NLPAPPL1}
\partial_t\Xi_A(t)=\Phi^A_{,i\alpha}(\tilde{f})\,\partial_{\alpha}v_i.\\
\end{equation}

%%\par\smallskip

Thus, using \eqref{ENTDEF}, \eqref{DIVPRODUCTS}$_1$ and
\eqref{NLPAPPL1}, we compute
\begin{equation*}
\begin{aligned}
\partial_t \bigl(\eta(\Theta)\bigr) & = V_i \partial_t{V_i}+G_{,A}(\Xi) \partial_t\Xi_A\\[0pt]
&=(V_i-v_i) \partial_t V_i+(G_{,A}(\Xi)-G_{,A}(\xi))
\partial_t \Xi_A +
\mathrm{div}\bigl({v_i g_i(\xi,\tilde f)}\bigr)\\[0pt]
&= \frac{1}{h}\sum_{j=1}^{\infty} \Chi^j(t)
\bigl(\nabla{\eta(\Theta)}-\nabla{\eta(\theta)}\bigr)
\delta\Theta^j + \mathrm{div}\bigl({v_i g_i(\xi,\tilde f)}\bigr).
\end{aligned}
\end{equation*}
Furthermore, by \eqref{DIVPRODUCTS}$_2$ we have
%%\begin{equation*}
%%\begin{aligned}
$
\partial_t\bigl(\bar{V}_i(V_i-\bar{V}_i)\bigr)
= \partial_t\bar{V}_i (V_i-\bar{V}_i) + \bar{V}_i
\partial_t V_i - \bar{V}_i \hspace{1pt}
\partial_t \bar{V}_i
%%\\&
=\partial_t \bar{V}_i (V_i-\bar{V}_i) +
\mathrm{div}\bigl(\bar{V}_ig_i(\xi,\tilde{f})\bigr) -
\nabla{\bar{V}_i} g_i(\xi,\tilde{f}) - \tfrac{1}{2}
\partial_t \bar{V}^2
$
%%\end{aligned}
%%\end{equation*}
while using \eqref{NLPAPPL1} we obtain
\begin{equation*}
\begin{aligned}
\partial_t (G_{,A}(\bar{\Xi})(\Xi-\bar{\Xi})_A)&= \partial_t (G_{,A}(\bar{\Xi}))(\Xi-\bar{\Xi})_A +
G_{,A}(\bar{\Xi}) \partial_t \Xi_A  - \partial_t (G(\bar{\Xi}))\\
&= \partial_t (G_{,A}(\bar{\Xi}))(\Xi-\bar{\Xi})_A + \nabla{v_i}
g_i(\bar{\Xi},\tilde{f})-
\partial_t (G(\bar{\Xi})).
\end{aligned}
\end{equation*}
Next, notice that by \eqref{GCOMPFLD} and \eqref{RENTFLUX} we have
\begin{equation}\label{RENTFULX1}
    q^r = v_i g_i(\xi,\tilde{f}) - \bar{V}_i g_i(\xi,\tilde{f}) -
    v_i g_i(\bar{\Xi},\tilde{f}) + \bar{V}_i
    g_i(\bar{\Xi},\tilde{f}).
\end{equation}
Hence by \eqref{ENTDEF}, \eqref{RENTDEF}, \eqref{DTERM},
\eqref{DIVPRODUCTS} and the last four identities we obtain
\begin{equation}\label{RELENTID1}
\partial \eta^r - \mathrm{div} \, q^r =
-\frac{1}{h}\sum_{j=1}^{\infty} \Chi^j(t) D^j + J
\end{equation}
where
\begin{equation*}
\begin{aligned}
J:= & - \mathrm{div}\bigl(\bar{V}_i g_i(\bar{\Xi},\tilde{f})\bigr)+
\nabla{\bar{V}_i} g_i(\xi,\tilde{f})
%%\\[2pt] &
+ \mathrm{div}\bigl(v_i g_i(\bar{\Xi},\tilde{f})\bigr) -
\nabla{v_i}
g_i(\bar{\Xi},\tilde{f}) \\
%%\\[2pt] &
& - \partial_t \bar{V}_i (V_i-\bar{V}_i) -
\partial_t (G_{,A}(\bar{\Xi}))(\Xi-\bar{\Xi})_A.
\end{aligned}
\end{equation*}

%%\par\smallskip

Consider now the term $J$. From \eqref{EXTSYS2}, \eqref{FGRADPROP}
and Lemma \ref{DIVPHIZLMM} it follows that
%%\begin{equation*}
%%\begin{aligned}
$
\partial_t \bar{V}_i = \Phi^A_{,i\alpha}(\bar{F})
\partial_{\alpha} (G_{,A}(\Bar{\Xi})),$
%%\\[3pt]
$
\partial_t (G_{,A}(\bar{\Xi})) = G_{,AB}(\Bar{\Xi}) \Phi^{B}_{,i\alpha}(\bar{F})
\partial_{\alpha} \bar{V}_i.
$
%%\end{aligned}
%%\end{equation*}
Then, \eqref{DIVPRODUCTS}$_{3,4}$ along with the last two
identities and the fact that $G_{,AB}=G_{,BA}$ implies
\begin{equation}\label{JTERM}
\begin{aligned}
J&= \partial_{\alpha}\bar{V}_i
\Bigl(g_{i\alpha}(\xi,\tilde{f})-g_{i\alpha}(\bar{\Xi},\tilde{f})\Bigr)\\[0pt]
&\quad+\partial_{\alpha}(G_{,A}(\bar{\Xi}))\Bigl(\Phi_{,i\alpha}^A(\tilde{f})(v_i-\bar{V}_i)-\Phi_{,i\alpha}^A(\bar{F})(V_i-\bar{V}_i)\Bigr)\\[0pt]
&\quad-G_{,AB}(\bar{\Xi})(\Xi-\bar{\Xi})_A
\Phi_{,i\alpha}^B(\bar{F})\,
\partial_{\alpha}{\bar{V}}_i\\[0pt]
&=\partial_{\alpha}\bar{V}_i \Bigl(g_{i\alpha}(\xi,\tilde{f})-g_{i\alpha}(\bar{\Xi},\tilde{f}) - g_{i\alpha}(\Xi,\bar{F}) + g_{i\alpha}(\bar{\Xi},\bar{F}) \Bigr)\\[0pt]
&\quad+\partial_{\alpha}(G_{,A}(\bar{\Xi}))\Bigl(\Phi_{,i\alpha}^A(\tilde{f})(v_i-\bar{V}_i)-\Phi_{,i\alpha}^A(\bar{F})(V_i-\bar{V}_i)\Bigr)\\[0pt]
&\quad+\partial_{\alpha}\bar{V}_i\Bigl(G_{,A}(\Xi)-G_{,A}(\bar{\Xi})-G_{,AB}(\bar{\Xi})(\Xi-\bar{\Xi})_B\Bigr)\Phi_{,i\alpha}^A(\bar{F})\\
&=:J_1+J_2+J_3.
\end{aligned}
\end{equation}
Using \eqref{GCOMPFLD} we rearrange the term $J_1$ as follows:
\begin{equation}\label{J1}
\begin{aligned}
J_1  &= \partial_{\alpha}\bar{V}_i
\Bigl[ \bigl(G_{,A}(\xi)\!-\! G_{,A}(\bar{\Xi})\bigr) \Phi^A_{,i\alpha}(\tilde{f}) \!-\! \bigl(G_{,A}(\Xi)\!-\! G_{,A}(\bar{\Xi})\bigr) \Phi^A_{,i\alpha}(\bar{F})\Bigr]\\[0pt]
&= \partial_{\alpha}\bar{V}_i \Bigl[ \bigl(G_{,A}(\xi)\!-\!
G_{,A}(\Xi)\bigr) \bigl( \Phi^A_{,i\alpha}(\tilde{f}) \!-\!
\Phi^A_{,i\alpha}(F) \bigr ) \\
%%&\!\qquad\qquad
& \quad +\! \bigl(G_{,A}(\xi)\!-\! G_{,A}(\Xi)\bigr) \bigl(
\Phi^A_{,i\alpha}(F) \!- \!\Phi^A_{,i\alpha}(\bar{F}) \bigr )
%%\\ &\!\qquad\qquad
\!+ \!\bigl(G_{,A}(\xi)\!- \!G_{,A}(\Xi)\bigr)
\Phi^A_{,i\alpha}(\bar{F})\\
%%&\!\qquad\qquad
& \quad + \!\bigl(G_{,A}(\Xi)\!-\! G_{,A}(\bar{\Xi})\bigr)
\bigl(
\Phi^A_{,i\alpha}(\tilde{f}) \!-\! \Phi^A_{,i\alpha}(F) \bigr ) \\
%%\\ &\!\qquad\qquad
& \quad + \!\bigl(G_{,A}(\Xi)\!-\! G_{,A}(\bar{\Xi})\bigr)
\bigl( \Phi^A_{,i\alpha}(F) \!-\! \Phi^A_{,i\alpha}(\bar{F}) \bigr )
\Bigr].
\end{aligned}
\end{equation}
We also modify the term $J_2$ writing it in the following way:
\begin{equation}\label{J2}
\begin{aligned}
J_2 & = \partial_{\alpha}(G_{,A}(\bar{\Xi})) \Bigl[
\Phi_{,i\alpha}^A(\tilde{f}) (v_i-\bar{V}_i) -
\Phi_{,i\alpha}^A(\bar{F})(V_i-\bar{V}_i) \Big]\\[0pt]
& = \partial_{\alpha}(G_{,A}(\bar{\Xi})) \Bigl[
\bigl(\Phi_{,i\alpha}^A(F)-\Phi_{,i\alpha}^A(\bar{F}) \bigr)
\bigl(V_i-\bar{V}_i\bigr) \\
%%&\! \quad \qquad\qquad\qquad
& \quad +
\bigl(\Phi_{,i\alpha}^A(\tilde{f})-\Phi_{,i\alpha}^A(F) \bigr)
\bigl(V_i- \bar{V}_i\bigr)
%%\\ &\! \quad \qquad\qquad\qquad
+
\bigl(\Phi_{,i\alpha}^A(\tilde{f})-\Phi_{,i\alpha}^A(F) \bigr)
\bigl(v_i- V_i\bigr) \\
%%&\!\quad \qquad\qquad\qquad
& \quad +
\bigl(\Phi_{,i\alpha}^A(F)-\Phi_{,i\alpha}^A(\bar{F}) \bigr)
\bigl(v_i- V_i\bigr)
%%\\ &\! \quad \qquad\qquad\qquad
+ \Phi_{,i\alpha}^A(\bar{F})
\bigl(v_i- V_i\bigr) \Bigr].
\end{aligned}
\end{equation}
By \eqref{JTERM}--\eqref{J2} we have $ J = J_1 + J_2 + J_3 = Q+S$.
Hence by \eqref{RELENTID1} we get \eqref{RENTID}.
\qed\end{proof}

\section{Proof of the Main Theorem}

The identity \eqref{RENTID} is central to our paper. In this
section, we estimate each of its terms and complete the proof via
Gronwall's inequality.

\subsection{A function $d(\cdot,\cdot)$ equivalent to the relative entropy }

\begin{definition*}
Let $\Theta_1=(V_1,\Xi_1),\Theta_2=(V_2,\Xi_2) \in\RR^{22}$. We
set
\begin{equation}\label{DISTDEF}
d(\Theta_1,\Theta_2) =
\BBR{1+|F_1|^{p-2}+|F_2|^{p-2}}\BBA{F_1-{F_2}}^2+\BBA{\Theta_1-{\Theta_2}}^2
\end{equation}
where $(F_1,Z_1,w_1)=\Xi_1,(F_2,Z_2,w_2)=\Xi_2\in \RR^{19}$.
\end{definition*}

%%\par\smallskip

The goal of this section is to show that the relative entropy
$\eta^r$ can be equivalently represented by the function
$d(\cdot,\cdot)$. Before we establish this relation, we prove an
elementary lemma used in our further calculations:
\begin{lemma}\label{AVESTLMM}
Assume $q \geqslant 1$. Then for all $u,v\in \RR^n$ and
$\bar{\beta} \in [0,1]$
\begin{equation}\label{AVEST2}
\int_{0}^{\bar{\beta}}\int_{0}^{1}(1-\beta)\BBA{u+\alpha(1-\beta)(v-u)}^q
d\alpha \,d\beta \, \geqslant\, c' \bar{\beta}  \bigl(|u|^q+|v|^q
\bigr)
\end{equation}
with constant $c'>0$ depending only on $q$ and $n$.
\end{lemma}

\begin{proof} Observe first that
\begin{equation} \label{AVEST1}
\int_{0}^{1}|u+\alpha(v-u)| \, d\alpha\,\geqslant\,
\bar{c}\BBR{|u|+|v|}, \quad \forall u,v \in \RR^n
\end{equation}
with  $\bar{c}=\frac{1}{4\sqrt{n}}$. Then, applying Jensen's
inequality and using \eqref{AVEST1}, we get
\begin{equation*}
\begin{aligned}
&\int_{0}^{\bar{\beta}}\int_{0}^{1}(1-\beta)\bigl|u+\alpha(1-\beta)(v-u)\bigr|^q d\alpha\,d\beta\\
& \geqslant \,\int_{0}^{\bar{\beta}}(1-\beta)\biggl(\int_{0}^{1}\bigl|u+\alpha\bigl((1-\beta)v+\beta u-u\bigr)\bigr|\,d\alpha\biggr)^q d\beta\\
&\geqslant \, \bar{c}^q  \int_{0}^{\bar{\beta}}(1-\beta)  \bigl( { |u|+|(1-\beta)v+\beta u}| \bigr) ^q d\beta\\
& \geqslant \, \frac{\bar{c}^q}{2} \bigl(|u|^q+|v|^q\bigr) \int_{0}^{\bar{\beta}}(1-\beta)^{q+1} \,d\beta.\\
\end{aligned}
\end{equation*}
Since $q\geqslant 1$ and $(1-\bar{\beta})\in [0,1]$, we have
%%\begin{equation*}
%%\begin{aligned}
$
\int_{0}^{\bar{\beta}}(1-\beta)^{q+1}d\beta \, = \, \frac{1 -
(1-\bar{\beta})^{q+2} }{q+2} \, \geqslant \,
\frac{\bar{\beta}}{q+2}.
$
%%\end{aligned}
%%\end{equation*}
Combining the last two inequalities we obtain \eqref{AVEST2}.
\qed\end{proof}

%%\par\smallskip

\begin{lemma}[$\eta^r$-equivalence]\label{RENTEQUIVLMM}
There exist constants $\mu,\mu'>0$ such that
\begin{equation}\label{RENTEQUIVD}
\begin{aligned}
\mu \, d(\Theta_1,\Theta_2) \, \leqslant
\,\eta^r(\Theta_1,\Theta_2) \, \leqslant \, \mu'
d(\Theta_1,\Theta_2)
\end{aligned}
\end{equation}
for every $\Theta_1=(V_1,\Xi_1), \Theta_2=(V_2,\Xi_2)
\in\RR^{22}$.
\end{lemma}

\begin{proof}
Notice that
\begin{equation}\label{RENTEST1}
\begin{aligned}
\eta^r( \Theta_1,\Theta_2)&=\eta  (\Theta_1) -\eta({\Theta_2})-\nabla\eta({\Theta_2})(\Theta_1-{\Theta_2})\\
&=\int_{0}^{1}
  \int_{0}^{1} s(\Theta_1-{\Theta_2})^T \bigl(\nabla^2\eta(\hat{\Theta})\bigr)(\Theta_1-{\Theta_2})\,ds\, d\tau.\\
\end{aligned}
\end{equation}
where
%%\begin{equation*}
$
\hat{\Theta}=({\hat{V},\hat{\Xi}})=(\hat{V},\hat{F},\hat{Z},\hat{w})
:={\Theta_2}+\tau s(\Theta_1-\Theta_2),$
%%\quad
$\tau,s \in [0,1].$
%%\end{equation*}
Observe next that
\begin{equation}\label{GRADG}
\nabla_{\Xi} G =
\begin{bmatrix}
\nabla_{F} H \quad {0} \quad  {0}
\end{bmatrix} + \nabla_{\Xi} R
\end{equation}
and therefore by \eqref{ENTDEF}
\begin{equation}\label{ETAHESS}
\begin{aligned}
&(\Theta_1\!-\!\Theta_2)^T \nabla^2\eta(\hat{\Theta})
(\Theta_1\!-\!\Theta_2)\\
&\!=\!|V_1\!-\!V_2|^2\! +\!(\Xi_1\!-\!\Xi_2)^T\nabla^2R(\hat{\Xi})(\Xi_1\!-\!\Xi_2)
%%\\ &
\!+\!(F_1\!-\!F_2)^T\nabla^2H(\hat{F})(F_1\!-\!F_2).
\end{aligned}
\end{equation}
Then (H1), \eqref{RENTEST1} and \eqref{ETAHESS} imply
\begin{equation}\label{RENTEST2}
\begin{aligned}
\tfrac{1}{2}\BBA{V_1-V_2}^2+\tfrac{\gamma}{2}\BBA{\Xi_1-\Xi_2}^2
&+\kappa \BBA{F_1-F_2}^2 \int_{0}^{1}\int_{0}^{1}s|\hat{F}|^{p-2}ds\,d\tau\\
\leqslant\,\eta^r( & \Theta_1 ,\Theta_2)\; \leqslant \\
\tfrac{1}{2}\BBA{V_1-V_2}^2+\tfrac{\gamma'}{2}\BBA{\Xi_1-\Xi_2}^2
&+\kappa' \BBA{F_1-F_2}^2 \int_{0}^{1}\int_{0}^{1}s|\hat{F}|^{p-2}ds\,d\tau.\\
\end{aligned}
\end{equation}
We now consider the integral term in \eqref{RENTEST2}. Recall that
$\hat{F} = F_2 + \tau s (F_1 - F_2)$. Then, estimating from above,
we get
\begin{equation*}
\int_{0}^{1}\int_{0}^{1}s|\hat{F}|^{p-2}ds\,d\tau \leqslant
2^{p-3}\BBR{|F_1|^{p-2}+|{F_2}|^{p-2}}
\end{equation*}
while for the estimate from below we use Lemma $\ref{AVESTLMM}$
(with $s=1-\beta$ and $\bar{\beta}=1$) and obtain
\begin{equation*}
\int_{0}^{1}\int_{0}^{1}s|\hat{F}|^{p-2}ds\,d\tau \geqslant \,
c'\BBR{|F_1|^{p-2}+|{F_2}|^{p-2}}.
\end{equation*}
Combining \eqref{RENTEST2} with the two last inequalities we
obtain \eqref{RENTEQUIVD}.
\qed\end{proof}

%%\par\smallskip

Observe that the smoothness of $\Bar{\Theta}$ implies that %%% here alteration:
there exists $M=M(T)>0$ such that
\begin{equation}\label{MBOUND}
\begin{aligned}
\quad M \, \geqslant  \, |\bar{\Theta}| + |\nabla_x\Bar{\Theta}| +
|\partial_t\Bar{\Theta}|, \quad  (x,t) \in \mathbb{T}^3 \times
[0,T].
\end{aligned}
\end{equation}

%%\par\smallskip

\begin{lemma}[$\mathcal{E}$-equivalence]\label{RENTINTEQUIVLMM}
%%% here alteration:
The relative entropy $\eta^r$ and function $d$ satisfy \begin{equation*}
\eta^r(\Theta,\bar{\Theta}), \, d(\Theta,\bar{\Theta})\in L^{\infty}\BBR{[0,T];L^1}.
\end{equation*}
Moreover,
\begin{equation*}
\mu \, \mathcal{E}(t) \, \leqslant  \int_{\mathbb{T}^3} \,
\eta^r\bigl(\Theta(x,t),\bar{\Theta}(x,t)\bigr)\,dx \, \leqslant
\, \mu' \mathcal{E}(t), \quad  t\in[0,T]
\end{equation*}
where
\begin{equation*}
\mathcal{E}(t):=\int_{\mathbb{T}^3}
d\bigl(\Theta(x,t),\bar{\Theta}(x,t)\bigr) \,dx
\end{equation*}
and constants $\mu,\mu'>0$ are defined in Lemma \ref{RENTEQUIVLMM}.

\end{lemma}
\begin{proof}
Fix $t\in[0,T]$. Then there exists $j\geqslant 1$ such that $t\in I_j$. Hence \eqref{CONTINTP},
\eqref{DISTDEF}, \eqref{MBOUND} and (H2) imply for $p\in{[6,\infty)}$
\begin{equation}\label{DISTEST1}
\begin{aligned}
d(\Theta(\cdot,t),\bar{\Theta}(\cdot,t))
&\leqslant C \Bigl( 1+|F|^p + |Z|^2+|w|^2+|V|^2\Bigr)\\
&\leqslant C \Bigl( 1+G(\Xi^{j-1}) + G(\Xi^{j})+|v^{j-1}|^2 +
|v^j|^2\Bigr)
\end{aligned}
\end{equation}
with $C=C(M)>0$ independent of $h$, $j$ and $t$. Hence
\eqref{ITERBOUND} and \eqref{DISTEST1} imply
\begin{equation}\label{DISTEST2}
\int_{\mathbb{T}^3} \, d(\Theta(\cdot,t),\bar{\Theta}(\cdot,t))\,
dx \, \leqslant C'(1+E_0), \quad \forall t\in [0,T]
\end{equation}
for some $C'=C'(M)>0$. Then \eqref{RENTEQUIVD} and
\eqref{DISTEST2} imply the lemma.
\qed\end{proof}

%%\par\smallskip

\subsection{Estimate for {the} term $Q$ on $t\in[0,T]$}
\begin{lemma}[$Q$-bound]\label{QTERMESTLMM} There exists $ \lambda=\lambda(M)
> 0$ such that
\begin{equation}\label{QTERMEST}
%%\begin{aligned}
%%\qquad \qquad \qquad \qquad\quad
\BBA{Q(x,t)} \leqslant \lambda \,
d(\Theta,\bar{\Theta}), \quad (x,t) \in \mathbb{T}^3 \times [0,T]
%%\end{aligned}
\end{equation}
where the term $Q$ is defined by \eqref{QTERM}.
\end{lemma}

\begin{proof} Let $C=C(M)>0$ be a generic constant. Notice
that %%% here alteration:
for all $F_1, F_2 \in \mdd{3}$
\begin{equation}\label{PHIDIFFPROP}
\bigl|\Phi_{,i\alpha}^A(F_1)-\Phi_{,i\alpha}^A(F_2)\bigr|\,\leqslant\,\left\{
\begin{aligned}
&\,0,& &A=1,\dots,9\\
&|F_1-F_2|,& & A=10,\dots,18\\
&3\bigl(|F_1|+|F_2|\bigr)|F_1-F_2|, & & A=19
\end{aligned}
\right.
\end{equation}
and hence
\begin{equation} \label{PHIDIFFPROP2}
|\Phi_{,i\alpha}^A(F)-\Phi_{,i\alpha}^A(\bar{F})|\,\leqslant\,C\BBR{1+|F|}\BBA{F-\bar{F}},\quad
A=1,\dots19.
\end{equation}
Then, using \eqref{MBOUND} and \eqref{PHIDIFFPROP2} we estimate
the first term of $Q$:
\begin{equation}\label{QTEST1}
%%\begin{aligned}
\bigl|\partial_{\alpha}  (G_{,A}(\Bar{\Xi}))
(\Phi_{,i\alpha}^A(F)-\Phi_{,i\alpha}^A(\Bar{F}))(V_i-\bar{V}_i)\bigr|
%%\\[2pt] &\qquad \qquad
\leqslant C \big(
(1+|F|^2)|F-\bar{F}|^2+|V-\bar{V}|^2 \big).
%%\end{aligned}
\end{equation}

Observe now that \eqref{GRADG} and \eqref{PHIDIFFPROP}$_1$ imply
for all $\Xi_1,\Xi_2\in\RR^{22}$, $F_3,F_4 \in \RR^{9}$
\begin{equation}\label{GPHIDIFF}
\begin{aligned}
&(G_{,A}(\Xi_1) - G_{,A}(\Xi_2))(\Phi_{,i\alpha}^A(F_3) -
\Phi_{,i\alpha}^A(F_4))\\
&=(R_{,A}(\Xi_1)-R_{,A}(\Xi_2))(\Phi_{,i\alpha}^A(F_3)-\Phi_{,i\alpha}^A(F_4)).
\end{aligned}
\end{equation}
Thus, by (H1), \eqref{PHIDIFFPROP2} and \eqref{GPHIDIFF} we obtain
the estimate for the second term:
\begin{equation}\label{QTEST2}
%%\begin{aligned}
\bigl|\partial_{\alpha} \Bar{V_i}
(G_{,A}(\Xi)\!-\!G_{,A}(\bar{\Xi}))(\Phi_{,i\alpha}^A(F)\!-\!\Phi_{,i\alpha}^A(\bar{F}))\bigr|
%%\\[2pt] &
\!\leqslant\! C \big(|\Xi\!-\!\bar{\Xi}|^2\!+\!(1\!+\!|F|^2)|F\!-\!\bar{F}|^2\big).
%%\end{aligned}
\end{equation}

%%\par\smallskip

Finally, we define for each $A=1,\dots,19$
\begin{equation}\label{JADEF}
\begin{aligned}
J_A&:=G_{,A}(\Xi)-G_{,A}(\bar{\Xi})-G_{,AB}(\bar{\Xi})\BBR{\Xi-\bar{\Xi}}_B\\
&\phantom{:}=\int_{0}^{1}\int_{0}^{1}s(\Xi-\bar{\Xi})^T \nabla^2 G_{,A}
(\hat{\Xi})(\Xi-\bar{\Xi}) \, ds \,d\tau
\end{aligned}
\end{equation}
where
%%\begin{equation*}
%%\begin{aligned}
$
\hat{\Xi} = (\hat{F},\hat{Z},\hat{w}):= \bar{\Xi}+\tau s (\Xi
-\bar{\Xi}),$
%%\quad
$\tau,s\in[0,1].$
%%\end{aligned}
%%\end{equation*}
By \eqref{GDECOMP} and (H5) we have for each $A=1,\dots,19$
\begin{equation}\label{GAHESSPROP}
\begin{aligned}
\bigl|(\Xi-\bar{\Xi})^T \nabla^2 G_{,A}(\hat{\Xi})(\Xi-\bar{\Xi})\bigr| \leqslant C
\big(|F-\Bar{F}|^2|\hat{F}|^{p-3} + |\Xi - \Bar{\Xi}|^2 \big).
\end{aligned}
\end{equation}
Then by \eqref{MBOUND} and \eqref{JADEF}, \eqref{GAHESSPROP} we
obtain the estimate for the third term:
\begin{equation}\label{QTEST3}
\begin{aligned}
|\partial_{\alpha}\bar{V}_i \, \Phi_{,i\alpha}^A(\bar{F}) \,
J_A|
&\leqslant C\Bigl(|\Xi\!-\!\bar{\Xi}|^2\!+\!|F\!-\!\bar{F}|^2 \!\!\int_{0}^{1}\!\!\!\int_{0}^{1}\!\!
|\bar{F}\! + \!\tau s(F \!-\!\bar{F})|^{p-3}ds\,d\tau \Bigr)\\
&\leqslant C\Bigl(|\Xi\!-\!\bar{\Xi}|^2\!+\!|F\!-\!\bar{F}|^2(1\!+\!|F|^{p-3})\Bigr).
\end{aligned}
\end{equation}

%%\par\smallskip

Thus by \eqref{DISTDEF}, \eqref{QTEST1}, \eqref{QTEST2} and \eqref{QTEST3} we conclude
for $p \in {[6,\infty)}$
\begin{equation*}
\BBA{Q(x,t)} \,\leqslant\,C
\Bigl(|\Theta-\bar{\Theta}|^2+(1+|F|^{p-2})|F-\bar{F}|^2\Bigr)
\leqslant C\,d(\Theta,\Bar{\Theta}). \tag*{\QED}
\end{equation*}
%\qed
\end{proof}

\subsection{Estimates for the terms $D^j$ and $S$ on $t \in I_j'\subset [0,T]$}
In this section, we consider $j \geqslant 1$ such that $(j-1)h <
T$ and estimate the dissipative and error terms for $t\in I'_j$
where
\begin{equation*}
I_j'\,:= I_j\bigcap \;[0,T] = [(j-1)h,jh)\bigcap \;[0,T].
\end{equation*}

\begin{lemma}[$D^j$-bound]\label{DTERMESTLMM} Let $D^j$ be the term defined
by \eqref{DTERM}. Then
\begin{equation}\label{DINTEGRB}
D^j \in L^{\infty}\bigl(I_j'\,; L^1(\mathbb{T}^3)\bigr)
\end{equation}
and %%% here alteration:
there exists constant $C_D>0$ independent of $h$ and $j$  such that %%% here alteration:
 for all times $\tau\in \bar{I}_j':=[(j-1)h,jh]\bigcap\,[0,T]$
\begin{equation} \label{DINTXTjEST}
%%\begin{aligned}
\int_{(j\!-\!1)h}^{\tau} \! \int_{\mathbb{T}^3} \!\!
\biggl(\frac{1}{h}D^j \!\biggr) \, dx \, dt
%%\\ &
\!\geqslant \!  a(\!\tau\!)     C_D \!\!\int_{\mathbb{T}^3}\!\!
|{\delta\Theta}^j|^2 \!+\!\bigl(|F^{j-1}|^{p-2}\!+\!|F^j|^{p-2} \bigr)
|\delta F^j|^2 \,dx \! \geqslant \! 0
%%\end{aligned}
\end{equation}
with
\begin{equation}\label{ADEF}
\begin{aligned}
\quad a(\tau):=\frac{\tau-h(j-1)}{h}\in [0,1], \quad \tau \in
\bar{I}_j'.
\end{aligned}
\end{equation}
\end{lemma}

\begin{proof}
By (H1), \eqref{ENTDEF} and the definition of $D^j$ we have for
$t\in I'_j$
\begin{equation} \label{DDISSP}
\begin{aligned}
D^j =\BBR{v-V} {\delta v^j} +\bigl(\nabla H(f)-\nabla
H(F)\bigr){\delta F^j}
+\bigl(\nabla R(\xi)-\nabla R(\Xi)\bigl) {\delta\Xi^j}.\\
\end{aligned}
\end{equation}
Consider each of the three terms in \eqref{DDISSP}. Notice that,
by \eqref{CONTINTP}, \eqref{CONSTINTP}, we have
\begin{equation}\label{DIFFAPPROX1}
\begin{aligned}
v(\cdot,t)-V(\cdot,t)&=\BBR{1-a(t)}\delta v^j \\
\xi(\cdot,t)-\Xi(\cdot,t)&=\BBR{1-a(t)}\delta \Xi^j.
\end{aligned}
\end{equation}
Using \eqref{DIFFAPPROX1} we compute
\begin{equation}\label{VRHDISSP}
\begin{aligned}
\bigl( v-V \bigr) {\delta v^j} &= \BBR{1-a(t)}|\delta
v^j|^2 \\
\bigl(\nabla R (\xi)-\nabla R (\Xi)\bigr) {\delta\Xi^j} &=
\BBR{1-a(t)}\int_{0}^{1} (\delta\Xi^j)^T \nabla^2R
(\Hat{\Xi})\,(\delta\Xi^j)
\,ds\\
\bigl(\nabla H (f)-\nabla H(F)\bigr) \delta F^j
&=\BBR{1-a(t)}\int_{0}^{1} (\delta
F^j)^{T}\nabla^2H(\hat{F})\,(\delta F^j) \,ds
\end{aligned}
\end{equation}
where
%%\begin{equation*}
$\Hat{\Xi}
=(\hat{F},\hat{Z},\hat{w}):=s{\xi}(\cdot,t)+(1-s)\Xi(\cdot,t), $
%%\quad
$s\in[0,1].$
%%\end{equation*}
Then (H1), \eqref{DDISSP} and~\eqref{VRHDISSP} together with the
fact that $\BBR{1-a(t)}\in[0,1]$ imply
\begin{equation}\label{DABSEST1T}
\begin{aligned}
\BBA{D^j(\cdot,t)} \leqslant \biggl( |\delta v^j|^2 +
\gamma'|\delta \Xi^j|^2 +\kappa'|\delta
F^j|^2\int_{0}^{1}|\Hat{F}(s,t)|^{p-2}ds\biggr).
\end{aligned}
\end{equation}
Consider now the two latter terms in \eqref{DABSEST1T}. Recalling
that $\Hat{F} = sf - (1-s)F$ and using (H2) together with
\eqref{CONTINTP}, \eqref{CONSTINTP} we obtain
\begin{equation*}
\begin{aligned}
&\gamma'|\delta \Xi^j|^2  + \kappa'|\delta F^j|^2\int_{0}^{1}|\Hat{F}(s,t)|^{p-2}ds \\
&\,\leqslant C \BBR{ 1 + |F^{j-1}|^p+|F^j|^p + |Z^{j-1}|^2+|Z^j|^2 + |w^{j-1}|^2+|w^j| }\\
\end{aligned}
\end{equation*}
for some $C>0$ independent of $h$, $j$ and $t$. Thus, combining
the last inequality with (H2), the growth estimate
\eqref{ITERBOUND} and \eqref{DABSEST1T}, we conclude %%% here is alteration
\begin{equation*}% \label{DABSINTXEST}
\begin{aligned}
\int_{\mathbb{T}^3}\,\BBA{D^j(x,t)}\,dx \, \leqslant \, \nu'
\bigl( 1+E_0 \bigr), \quad \forall t\in I_j'
\end{aligned}
\end{equation*}
for some $\nu'>0$ independent of $h$, $j$ and $t$. This proves
\eqref{DINTEGRB}.

%%\par\smallskip

Let us now estimate $D^j$ from below. By \eqref{DDISSP},
\eqref{VRHDISSP} and (H1) we obtain
\begin{equation}\label{DEST}
\begin{aligned}
D^j(\cdot,t)
& \, \geqslant \,\nu\BBR{1-a(t)} \Bigl(|\delta \Theta^j|^2+|\delta F^j|^2\int_{0}^{1}|\hat{F}(s,t)|^{p-2}ds \Bigr) \geqslant \,0\\
\end{aligned}
\end{equation}
for $\nu=\min(1,\gamma,\kappa)>0$. Notice that
%%\begin{equation*}
%%\begin{aligned}
$$
\hat{F}(s,t) =sf(t)+(1-s)F(t) = F^j+(1-s)(1-a(t))(F^{j-1}-F^j).
$$
%%\end{aligned}
%%\end{equation*}
Then, by making use of Lemma $\ref{AVESTLMM}$ we obtain for $\tau
\in \bar{I}'_j$
\begin{equation*}
\begin{aligned}
&\int_{(j-1)h}^{\tau} \Big(\BBR{1-a(t)} |\delta
F^j|^2\int_{0}^{1}|\Hat{F}(s,t)|^{p-2}ds \Big) \, dt\\
&=h
|\delta F^j|^2 \int_{0}^{a(\tau)}\int_{0}^{1}(1-\beta)
|F^j+\alpha(1-\beta)(F^{j-1}-F^j)|^{p-2} d\alpha\,d\beta \\
%%[3pt] &\quad
&\geqslant h \, a(\tau)\,c' \big(|F^{j-1}|^{p-2}+|F^{j}|^{p-2}\big) |\delta F^j|^2
\end{aligned}
\end{equation*}
where we used the change of variables $\alpha = 1-s$ and
$\beta=a(t)$. Similarly, we get
\begin{equation*}
\begin{aligned}
\int_{(j-1)h}^{\tau} \BBR{1-a(t)}|{\delta\Theta}^j|^2 \, dt \, =
\, h \,|{\delta\Theta}^j|^2 \int_{0}^{a(\tau)}(1-\beta) \, d\beta
\, \geqslant \, \frac{h \, a(\tau)}{2} |{\delta\Theta}^j|^2 .
\end{aligned}
\end{equation*}
Then \eqref{DEST} and the last two estimates imply
\eqref{DINTXTjEST} for $C_D=\min(\nu c',\frac{\nu}{2})
> 0$.
\qed\end{proof}

%%\par\smallskip

\begin{lemma}[$S$-bound]\label{STERMESTLMM} Let $S$ be the term defined
by \eqref{STERM}. Then
\begin{equation}\label{SINTEGRB}
S \in L^{\infty}\bigl(I_j'\,;L^1(\mathbb{T}^3)\bigr)
\end{equation}
and there exists constant s$C_S>0$ independent of $h$, $j$ such that for any
$\eps>0$ and all $\tau \in \bar{I}'_j$
\begin{equation}\label{SINTEST}
\begin{aligned}
&\int_{(j-1)h}^{\tau}  \int_{\mathbb{T}^3} \BBA{S(x,t)} \,dx\,dt\,\\
%%[2pt]
& \leqslant  C_S\biggl[ \, a(\tau)(h+\varepsilon)
\int_{\mathbb{T}^3} |{\delta\Theta}^j|^2
+(|F^{j-1}|^{p-2}+|F^{j}|^{p-2})|\delta F^j|^2 \, dx
\\
%%& \qquad\qquad
& \quad + \frac{a(\tau)h^2}{\eps}(3+2E_0)+\int_{(j-1)h}^{\tau}\int_{\mathbb{T}^3}
\,d(\Theta,\bar{\Theta}) \,dx \, dt \, \biggr]
\end{aligned}
\end{equation}
with $a(\tau)$ defined by \eqref{ADEF}.
\end{lemma}

\begin{proof}
As before, we let $C=C(M)>0$ be a generic constant and remind the
reader that all estimates are done for $t\in I'_j$.

%%\par\smallskip

Observe that \eqref{CONTINTP}$_2$, \eqref{CONSTINTP}$_3$ and
\eqref{ADEF} imply
%%\begin{equation*}
%%\begin{aligned}
$
F(\cdot,t)-\tilde{f}(\cdot,t)=a(t) \hspace{1pt} \delta F^j.
$
%%\end{aligned}
%%\end{equation*}
Hence by \eqref{CONTINTP}$_2$, \eqref{CONSTINTP}$_3$,
\eqref{PHIDIFFPROP}, \eqref{ADEF} and the identity above we get
the estimate
\begin{equation}\label{PHIDIFFPROP3}
%%\begin{aligned}
\bigl|\Phi_{,i\alpha}^{A}(\tilde{f})- \Phi_{,i\alpha}^{A}(F)
\bigr|  \leqslant C \bigl(1+|\tilde{f}|+|F| \bigr) |F-\tilde{f}|
%%\\ &
\leqslant C \BBR{1+|F^{j-1}|+|F^{j}|} |\delta F^j|.
%%\end{aligned}
\end{equation}
Thus \eqref{PHIDIFFPROP2}, \eqref{ADEF}, \eqref{DIFFAPPROX1}$_1$,
\eqref{PHIDIFFPROP3} and  Young's inequality imply
\begin{equation}\label{SESTTM1}
\begin{aligned}
&\bigl| \Phi_{,i\alpha}^A(\bar{F})(v_i - V_i)\bigr|
%%\\[1pt] &\,
 + \bigl|(\Phi_{,i\alpha}^{A} (F)-\Phi_{,i\alpha}^A(\bar{F}))(v_i-V_i)\bigr|\\
& +\bigl|(\Phi_{,i\alpha}^{A} (\tilde{f}) - \Phi_{,i\alpha}^{A}(F))(v_i-V_i) \bigr|
%%\\[1pt] &\,
+\bigl|(\Phi_{,i\alpha}^{A} (\tilde{f}) - \Phi_{,i\alpha}^A(F))(V_i-\bar{V}_i)\bigr| \\
%%[1pt] &\qquad\qquad\qquad \,
&\leqslant  C \Bigl( |\delta v^j|+
(1+|F|^2)|F-\bar{F}|^2 + |\delta v^j|^2
%%\\[0pt] &\quad \qquad\qquad\qquad\qquad
+ (1+|F^{j-1}|^2+|F^{j}|^2)
|\delta F^j|^2 \\
& \quad + |V-\bar{V}|^2  \Bigl).
\end{aligned}
\end{equation}
We also notice that for all $F_1, F_2 \in \mdd{3}$
\begin{equation*}
\begin{aligned}
H_{,i\alpha}(F_1)-H_{,i\alpha}(F_2)&=
\int_{0}^{1}\frac{\partial^{2}H}{\partial F_{i\alpha}\partial F_{lm}}\bigl(sF_1+(1-s)F_2\bigr)(F_{1}-F_{2})_{lm} \, ds.\\
\end{aligned}
\end{equation*}
Hence (H1), (H5), \eqref{ADEF}, \eqref{DIFFAPPROX1}$_2$ and the
identity above imply
\begin{equation}\label{SESTTM2}
\begin{aligned}
\bigl|\Phi_{,i\alpha}^A(\bar{F}) \BBR{G_{,A}(\xi)\!-\!G_{,A}(\Xi)}\bigr|
&\!\leqslant C\bigl(|\nabla H(f)\!-\!\nabla H(F)| \!+\!|\nabla R(\xi)\!-\!\nabla R(\Xi)| \bigr)\\
& \!\leqslant  C\Bigl(|f \!-\! F|\!\! \int_{0}^1 \!\!\!|sf \!+\!(1\!-\!s)F|^{p-2}ds
\!+\! |\xi\!-\!\Xi|\Bigr)\\
&\!\leqslant C\bigl((|F^{j-1}|^{p-2}\!+\!|F^j|^{p-2})|\delta F^j| \!+ \!|\delta\Xi^j| \bigr).
\end{aligned}
\end{equation}

%%\par\smallskip

Next,  by (H1),  \eqref{PHIDIFFPROP2}, \eqref{GPHIDIFF},
\eqref{ADEF}, \eqref{DIFFAPPROX1}$_2$ and \eqref{PHIDIFFPROP3} we
obtain
\begin{equation}\label{SESTTM3}
\begin{aligned}
&\bigl|  (G_{,A}(\xi)\!-\!G_{,A}(\Xi))      (\Phi_{,i\alpha}^A(\tilde{f})\!-\!\Phi_{,i\alpha}^A(F)) \bigr| \\
%%\\[1pt]
&+\bigl| (G_{,A}(\xi)\!-\!G_{,A}(\Xi))      (\Phi_{,i\alpha}^A(F)\!-\!\Phi_{,i\alpha}^A(\bar{F}))   \bigr|\\
%%[1pt]
&+\bigl| (G_{,A}(\Xi)\!-\!G_{,A}(\bar{\Xi}))(\Phi_{,i\alpha}^A(\tilde{f})\!-\!\Phi_{,i\alpha}^A(F)) \bigr|\\
%%[1pt] & \qquad \qquad \qquad \,
&\leqslant C \Bigl( |\delta \Xi^j|^2 \!+\! (1\!+\!|F^{j-1}|^2\!+\!|F^{j}|^2)|\delta F^j|^2
%%\\ & \qquad \qquad \qquad \qquad \,\, \qquad \quad
\!+\! (1\!+\!|F|^2)|F\! -\!\bar{F}|^2 \!+\! |\Xi \!-\! \bar{\Xi}|^2 \Bigr).
\end{aligned}
\end{equation}

%%\par\smallskip

Finally, \eqref{STERM}, \eqref{MBOUND}, and the estimates \eqref{SESTTM1}-\eqref{SESTTM3}
imply for $p\in{[6,\infty)}$
\begin{equation}\label{SEST1}
\begin{aligned}
|S(\cdot,t)| \leqslant & C_S \biggl[  (|F^{j-1}|^{p-2}+|F^{j}|^{p-2})|\delta F^j|^2 + |{\delta\Theta}^j|^2 \\
 &+(|F^{j-1}|^{p-2}+|F^{j}|^{p-2})|\delta F^j| + |\delta \Theta^j| + d(\Theta,\bar{\Theta}) \biggr]
\end{aligned}
\end{equation}
for some $C_S>0$ independent of $h$, $j$ and $t$. Then, by
\eqref{ITERBOUND} and \eqref{DISTEST1} we conclude that the right
hand side of \eqref{SEST1} is in $L^{\infty}\BBR{I_j'\,;
L^1(\mathbb{T}^3)}$ which proves \eqref{SINTEGRB}.

%%\par\smallskip

We now pick any $\eps>0$. Then, employing Young's inequality,
we obtain
%%\begin{equation*}\label{GADIFF3RHS1}
%%\begin{aligned}
$$
(|F^{j-1}|^{p-2}\!+ |F^j|^{p-2})|\delta F^j|
 \!\leqslant \!\frac{h}{\varepsilon}\! \BBR{|F^{j-1}|^{p-2}\!+\!|F^j|^{p-2}}
%% \\ &\qquad
\!+\frac{\varepsilon}{h}\!\BBR{|F^{j-1}|^{p-2}\!+\!|F^j|^{p-2}}|\delta F^j|^2
%%\\ \end{aligned} \end{equation*}
$$
and, similarly, $|\delta \Theta^j| \leqslant
\frac{h}{\varepsilon}+\frac{\varepsilon}{h} |\delta\Theta^j|^2$.
Thus \eqref{SEST1} and the last two estimates imply
\begin{equation}\label{SEST2}
\begin{aligned}
\BBA{S(\cdot,t)} \, \leqslant &\,  C_S
\Bigl[\bigl(1+\frac{\varepsilon}{h}\bigr)\bigl( \,
|{\delta\Theta}^j|^2 +(|F^{j-1}|^{p-2}+|F^{j}|^{p-2})|\delta
F^j|^2 \,\bigr)\\
&+\frac{h}{\varepsilon}\bigl(1+|F^{j-1}|^{p-2}+|F^j|^{p-2}\bigr) + d(\Theta,\bar{\Theta})\Bigr].
\end{aligned}
\end{equation}
To this end, we integrate %%% here alteration:
\eqref{SEST2} and use (H2) along with
\eqref{ITERBOUND} to get %%% here alteration:
\eqref{SINTEST}.
\qed\end{proof}

%%\par\smallskip

\subsection{Conclusion of the proof via Gronwall's inequality}

We now estimate the left hand side of the relative entropy
identity \eqref{RENTID}:

\begin{lemma}[LHS estimate]\label{LHSESTLMM} Let $\eta^r$, $q^r$ be the
relative entropy and relative entropy flux, respectively, defined
by \eqref{RENTDEF} and \eqref{RENTFLUX}. Then
\begin{equation}\label{RENTIDLHSINTEGRB}
\Bigl(\partial_t \eta^r-\mathrm{div}\,q^r \Bigr)  \in
L^{\infty}\BBR{[0,T], L^1(\mathbb{T}^3) }
\end{equation}
and %%% here alteration:
there exists $\bar{\eps}>0$ such that for all $h \in (0,\bar{\eps})$
and $\tau \in [0,T]$
\begin{equation}\label{RENTIDLHSEST}
\begin{aligned}
\int_{0}^{\tau}\int_{\mathbb{T}^3}
\,\bigl(\partial_t\,\eta^r-\mathrm{div}\,q^r \bigr)\,dx \,dt
\leqslant C_I\Bigl(\tau h +\int_{0}^{\tau}\int_{\mathbb{T}^3}
d(\Theta,\bar{\Theta}) \,dx\,dt\Bigr).
\end{aligned}
\end{equation}
for some constant $C_I=C_I(M,E_0,\bar{\eps})>0$.
\end{lemma}

\begin{proof} Lemma \ref{RENTEQUIVLMM}, \eqref{QTERMEST}, \eqref{DINTEGRB},
and \eqref{SINTEGRB} imply that the right hand side of the
relative entropy identity \eqref{RENTID} is in
$L^{\infty}\BBR{[0,T]; L^1(\mathbb{T}^3)}$. This proves
\eqref{RENTIDLHSINTEGRB}.

%%\par\smallskip

Notice that the constants $C_D$ and $C_S$ (that appear in Lemmas
\ref{DTERMESTLMM} and \ref{STERMESTLMM}, respectively) are
independent of $h$, $j$. Then set $\bar{\eps}:= \frac{C_D}{2C_S}$. Take
now $h\in(0,\bar{\eps})$ and $\tau\in[0,T]$. Using Lemmas
\ref{QTERMESTLMM}, \ref{DTERMESTLMM} and \ref{STERMESTLMM} (with
$\eps=\bar{\eps}$) along with the fact that
$-C_D+C_S(h+\bar{\eps}) \leqslant 0 \,$ we get
\begin{equation*}
\begin{aligned}
\int_{0}^{\tau}\hspace{-2pt}\int_{\mathbb{T}^3}
\Bigl(-\frac{1}{h}\sum_{j=1}^{\infty} \Chi^j(t) \hspace{1pt}
D^j+|S|+|Q|\,\Bigr)dx dt \, \leqslant\, C_I\Bigl(\tau h
+\int_{0}^{\tau} \hspace{-2pt}\int_{\mathbb{T}^3}
d(\Theta,\bar{\Theta}) dx dt\Bigr)
\end{aligned}
\end{equation*}
with $C_I:= 3\max \big(C_S \frac{1+E_0}{\bar{\eps}},C_S+\lambda \big)>0$. Hence by
\eqref{RENTID} and the estimate above we obtain
\eqref{RENTIDLHSEST}.
\qed\end{proof}

%%\par\smallskip

Observe that (P4), (P5), \eqref{RENTFLUX}, \eqref{TIMEDVVXI},
\eqref{gEST}, \eqref{DIVPRODUCTS}, and \eqref{RENTFULX1} imply
\begin{equation*}%\label{RENTFLUXINTEGRB}
\mathrm{div}\,q^r\in L^{\infty}\BBR{[0,T];L^1(\mathbb{T}^3)}
\end{equation*}
and hence by \eqref{RENTIDLHSINTEGRB}
\begin{equation}\label{RENTINTEGRB}
\partial_t \eta^r \in L^{\infty}\BBR{[0,T];L^1(\mathbb{T}^3)}.
\end{equation}

%%\par\smallskip

Take now arbitrary $h\in(0,\bar{\eps})$ and $\tau \in [0,T]$. Due
to periodic boundary conditions (by the density argument) we have
$\int_{\mathbb{T}^3} \bigl(\mathrm{div}\,q^r(x,s)\bigr) \,dx=0$
for a.e. $s\in[0,T]$ and hence
\begin{equation*}
\int_{0}^{\tau}\int_{\mathbb{T}^3} \mathrm{div}\,q^r \,dx\,dt=0.
\end{equation*}
Finally, by construction for each fixed $\bar{x}\in \mathbb{T}^3$
the function $\eta^r(\bar{x},t):[0,T] \to \RR$ is absolutely
continuous with the weak derivative $\partial_t\eta^r(\bar{x},t)$.
Then, by \eqref{RENTINTEGRB} and  Fubini's theorem we have
\begin{equation*}
\int_{0}^{\tau}\int_{\mathbb{T}^3}\partial_t \eta^r
dx\,dt=\int_{\mathbb{T}^3} \BBS{\int_{0}^{\tau} \partial_t
\eta^r(x,t) \,d\tau} dx =\int_{\mathbb{T}^3} \Bigl(\eta^r(x,\tau)
- \eta^r(x,0)\Bigr) dx.
\end{equation*}
Thus by Lemma \ref{RENTINTEQUIVLMM},
\eqref{RENTIDLHSINTEGRB}-\eqref{RENTINTEGRB} and the two
identities above we obtain
\begin{equation}\label{PSIGROWTH}
\mathcal{E}(\tau)\,\leqslant\,\bar{C}\Bigl(\mathcal{E}(0)+\int_{0}^{\tau}\mathcal{E}(t)\,dt+h\Bigr)
\end{equation}
with $\bar{C} \!:=\! \tfrac{T}{\mu}\max(C_I,\mu')$ independent of
$\tau, h$. Since $\tau\!\in\![0,T]$ is arbitrary, by~\eqref{PSIGROWTH} and  Gronwall's inequality we conclude
\begin{equation*}
\mathcal{E}(\tau)\,\leqslant\,\bar{C} \bigl(\mathcal{E}(0)+h\bigr)
 e^{\bar{C}T}, \quad \forall\tau\in [0,T].
\end{equation*}
In this case, if $\mathcal{E}^{(h)}(0)\to 0$ as $h\downarrow 0$,
then $\sup_{\tau\in[0,T]} \bigl(\mathcal{E}^{(h)}(\tau)\bigr) \to
0$, as $h\downarrow 0$. \par\bigskip

\subsection*{Acknowledgement}

Research partially supported by the EU FP7-REGPOT project ``Archimedes Center for Modeling, Analysis and Computation'' and the EU EST-project ``Differential Equations and Applications in Science and Engineering''.

\end{document}